\documentclass[11pt]{amsart}

\usepackage{amssymb,amsthm,amsfonts,latexsym,mathtools}

\usepackage{amsmath}
\usepackage{mathrsfs}
\usepackage{stmaryrd}
\usepackage[utf8]{inputenc}
\usepackage{color}
\usepackage{enumitem}
\usepackage{accents}
\usepackage{hyperref}
\usepackage[noabbrev]{cleveref}
\usepackage{tikz}
\usepackage{tikz-cd}
\usetikzlibrary{arrows.meta, positioning}
\usetikzlibrary{shapes,arrows}
\usepackage{pbox}

\newcommand{\lcm}{\mathrm{lcm}}
\newcommand{\LC}{\mathcal{L}}
\newcommand{\supp}{\mathrm{supp}}
  \makeatletter
  \newcommand{\numberlike}[2]{%
     \expandafter\def\csname c@#1\endcsname{%
         \expandafter\csname c@#2\endcsname}%
  }
  \makeatother

  \def\DefaultNumberTheoremWithin{section}

  \theoremstyle{plain}
  \newtheorem{lemma}{Lemma}
     \numberwithin{lemma}{\DefaultNumberTheoremWithin}
     \labelformat{lemma}{Lemma~#1}
  
     \numberwithin{claim}{\DefaultNumberTheoremWithin}
     \numberlike{claim}{lemma}
     \labelformat{claim}{Claim~#1}
  \newtheorem{theorem}{Theorem}
     \numberwithin{theorem}{\DefaultNumberTheoremWithin}
     \numberlike{theorem}{lemma}
     \labelformat{theorem}{Theorem~#1}
  \newtheorem{corollary}{Corollary}
     \numberwithin{corollary}{\DefaultNumberTheoremWithin}
     \numberlike{corollary}{lemma}
     \labelformat{corollary}{Corollary~#1}
     
  \newtheorem{proposition}{Proposition}
     \numberwithin{proposition}{\DefaultNumberTheoremWithin}
     \numberlike{proposition}{lemma}
     \labelformat{proposition}{Proposition~#1}

     \numberwithin{conjecture}{\DefaultNumberTheoremWithin}
     \numberlike{conjecture}{lemma}
     \labelformat{conjecture}{Conjecture~#1}

  \newtheorem*{warning*}{Warning}

  \theoremstyle{definition}
  \newtheorem{definition}{Definition}
     \numberwithin{definition}{\DefaultNumberTheoremWithin}
     \numberlike{definition}{lemma}
     \labelformat{definition}{Definition~#1}

  \theoremstyle{definition}
  
     \numberwithin{question}{\DefaultNumberTheoremWithin}
     \numberlike{question}{lemma}
     \labelformat{question}{Question~#1}

  \theoremstyle{definition}
  
     \numberwithin{problem}{\DefaultNumberTheoremWithin}
     \numberlike{problem}{lemma}
     \labelformat{problem}{Problem~#1}
  \theoremstyle{remark}
  \newtheorem{remark}{Remark}

     \numberwithin{remark}
     {\DefaultNumberTheoremWithin}
     \numberlike{remark}{lemma}
     \labelformat{remark}{Remark~#1}
  \theoremstyle{remark}

  \newtheorem{example}{Example}
     \numberwithin{example}{\DefaultNumberTheoremWithin}
     \numberlike{example}{lemma}
     \labelformat{example}{Example~#1}
    
\labelformat{equation}{(#1)}
\labelformat{figure}{Figure~#1}
\labelformat{table}{Table~#1}
\labelformat{chapter}{Chapter~#1}
\labelformat{appendix}{Appendix~#1}
\labelformat{section}{Section~#1}
\def\H{\mathcal {H}}

\title{{\parbox{\textwidth}{\centering 
 On Some Properties of LCM-Lattices of Edge Ideals of k-Uniform Hypergraphs}}}
\author[Muneeba]{Muneeba Mansha}
\address{
}
\author[S. Ahmad]{Sarfraz Ahmad}
\address{COMSATS University Islamabad (Lahore Campus), Lahore-Pakistan.}
\email{muneeba.math@gmail.com, Sarfrazahmad@cuilahore.edu.pk }
\begin{document}
\maketitle
\section{Abstract}
In this article, we investigate the combinatorial and algebraic properties of the lcm-lattice associated with the edge ideal of a hypergraph. Let $\H$ be a hypergraph, $I(\H)$ its corresponding edge ideal in a polynomial ring in $n$ variables, and $\mathrm{Icm}(I(\H))$ the associated lcm-lattice. We establish conditions under which the lcm-lattice of an edge ideal is Boolean, modular, or complemented. Furthermore, we extend these results to the case of the product of lcm-lattices in the complemented case. Additionally, we study the effects of polarization on the lcm-lattices of $I(\H)$ and its polarized ideal.\\
AMS Classification: 06A06, 06A07, 06A11.\\
Keywords: posets, Boolean, modular, complemented lattices.

\section{introduction}
  Hypergraphs and their associated algebraic structures have long been central objects of study in combinatorial commutative algebra. Given a hypergraph $\mathcal{H}$, its edge ideal $I(\mathcal{H})$, defined in a polynomial ring over $n$ variables, encodes the combinatorial structure of $\mathcal{H}$ in an algebraic form. The study of the $\lcm$-lattice, denoted by $\mathrm{Icm}(I(\mathcal{H}))$, provides a lattice-theoretic perspective on the generators of $I(\mathcal{H})$.

The structural properties of $\lcm$-lattices such as being Boolean, modular, or complemented which play a crucial role in determining the algebraic and combinatorial behavior of the underlying ideal. In this article, we identify precise conditions under which the $\lcm$-lattice of a hypergraph's edge ideal possesses these properties. We further investigate how these properties extend to the product of $\lcm$-lattices in the complemented case and analyze the effect of polarization on the $\lcm$-lattice structure.

Foundational background on monomial ideals can be found in the work of Herzog and Hibi \cite{HerzogHibi2011}. Classical results on posets and enumerative techniques, which are central to the study of monomial ideals and their associated lattices, are presented in Stanley’s monograph \cite{StanleyRP2011}. More recently, Dorang \cite{DorangMc2025, DorangMT2025} investigated structural properties of $\lcm$-lattices of monomial ideals. The concept of the $\lcm$-lattice of a monomial resolution was introduced by Gasharov, Peeva, and Welker \cite{WelkerPV1999}, who showed that if there exists a map between two $\lcm$-lattices that is bijective on atoms and preserves joins, then a resolution of one ideal induces a resolution of the other. In particular, if the map is an isomorphism, the two ideals share the same total Betti numbers and projective dimension. 

Kuei-Nuan Lin and Sonja Mapes \cite{KUEISONJA2022} established a connection between the dual hypergraph and the $\lcm$-lattice of a given square-free monomial ideal. Moreover, Kimura, Rinaldo, and Terai \cite{KIMURATERAI} showed that the projective dimension of a monomial ideal depends on the $1$-skeleton structure of its dual hypergraph. Some results on the homotopy of complemented lattices are also given in \cite{sara}.
  \section{Results and Discussions}
   This section presents the basic concepts, results, and discussions. Let $G = (V,E)$ be a simple undirected graph with vertex set $V = \{1, \dots, n\}$ and edge set $E$. More generally, a $k$-uniform hypergraph $\mathcal{H} = (V,E)$ consists of a vertex set $V = \{1, \dots, n\}$ and an edge set $E$, where each edge is a $k$-element subset of $V$. The corresponding edge ideal $I(\mathcal{H})$ is the monomial ideal in the polynomial ring $S = \mathbb{K}[x_1, \dots, x_n]$ over a field $\mathbb{K}$, generated by the square-free monomials 
\[
x_{i_1} \cdots x_{i_k} \quad \text{for each } \{i_1, \dots, i_k\} \in E
\]
(see \cite{LinMcCullough2013, GhoshPaulos2024} for details).
 

    

\begin{definition}
     A poset is a pair $(S, \preceq)$, where $S$ is a set and $\preceq$ is a binary relation on $S$ satisfying:
\begin{itemize}
      \item {Reflexivity:} $x \preceq x$ for all $x \in S$;
     \item {Antisymmetry:} if $x \preceq y$ and $y \preceq x$, then $x = y$;
     \item {Transitivity:} if $x \preceq y$ and $y \preceq z$, then $x \preceq z$.
\end{itemize}
If these conditions hold, the relation $\preceq$ is called a \emph{partial order}.
\end{definition} 
  
\begin{definition}
     A \emph{lattice} $\LC$ is a poset $(E, \leq)$ such that, for any two elements $a$ and $b$ in $\LC$, there exists a unique greatest lower bound, denoted by $a \wedge b$ and called the \emph{meet} of $a$ and $b$, and a unique least upper bound, denoted by $a \vee b$ and called the \emph{join} of $a$ and $b$.
\end{definition}
     A lattice can be visualized through its Hasse diagram, where elements are arranged according to their rank, i.e., smaller elements (lower rank) appear below larger ones (higher rank). Line segments are drawn to connect each element with those it directly covers. For a detailed introduction to lattices and order theory, we refer to \cite{BirkhoffG1967,DaveyPristley2002}. 
\begin{definition}
     The \emph{$\lcm$-lattice} of $I(G)$, denoted $\LC(I(G))$, is the lattice whose elements are all least common multiples of subsets of the minimal generators of $I(G)$, partially ordered by divisibility. We denote the unique maximal element (the least common multiple of all generators) by $\hat{1}$ and the minimal element by $\hat{0}$.
\end{definition}
%

%
\subsection{Hypergraphs and Boolean Lattices}
\begin{definition}\label{blndef}
    Let $\LC$ denote the set of all subsets of $[r]$, ordered by inclusion. Then $\LC$ is a lattice, called the \emph{Boolean lattice of rank $r$}, and is defined as
    \[\LC = \{ A \mid A \subseteq \{1,\ldots,r\} \},\]
    ordered by inclusion. 
\end{definition}

\begin{theorem}\label{Booleanhyper}
    Let \( \H = (V, E) \) be a $k$-uniform hypergraph, and let \( I(\H) = (x_{a_1}\cdots x_{a_k} \mid \{a_i,\ldots, a_k\} \in E(\H)) \) be its edge ideal with \( \LC(I(\H)) \) to be its $\lcm$-lattice. Then the following are equivalent:
\begin{itemize}
     \item[(a)] $\LC (I(\H))$ is isomorphic to a Boolean lattice
    \item[(b)] every hyperedge \( e \in E(\H) \) contains at least one vertex \( v_e \in e \) such that \( v_e \notin e' \) for all other hyperedges \( e' \in E(\H) \setminus \{e\} \).
\end{itemize}
\end{theorem}
\begin{proof}
    Let $\mathcal{P} (E(\H))$ be the power set of the edge set $E(\H)$ of the hypergraph $\H$. Let $f$ be a map from $\mathcal{P} (E(\H))$ to the elements of the $\lcm$-lattice \( \LC(I(\H)) \), defined as follows:
    \[
      \text{for}\,\,\,\, A \subseteq E(\H),\,\,\,\,\,\,\,\, A \mapsto \mathrm{lcm}(A).
    \]
    Then the $\lcm$-lattice \( \LC(I(\H)) \) is Boolean if and only if $f$ is bijective.\\
$a) \implies b)$\\
     Suppose \( \LC(I(\H)) \) is Boolean and let $e_1$ be a hyperedge in $E(\H)$ such that for each $v_{e_1}\in e_1$, there exist at least one hyperedge $e_2 \in E(\H)$ such that $v_{e_1}\in e_2$. Thus, every vertex of \( e_1 \) lies in some other hyperedge, hence
     \[
      \cup_{e \in E(\H)} e \;=\; \cup_{e \in E(\H) \setminus \{e_1\}} e,
     \]
     contradicting injectivity of the map $f$.\\
$b) \implies a)$\\
     Conversely, the map $f$ from $\mathcal{P} (E(\H))$ to the elements of the $\lcm$-lattice \( \LC(I(\H)) \) is surjective by definition. Now for injectivity, assume we have two subsets $A$ and $B$ from the edge set $E(\H)$ such that $A\not= B$. We need to show,
     \[
    f(A)\not= f(B).
    \] 
    To prove this, it is sufficient to show,
    \[
    \cup_{e \in A} e \;\not=\; \cup_{e \in B}e,
     \]
     i.e., $A\not= B$, so there exist at least one hyperedge $e \in A$, such that $e \not\in B$. Since every hyperedge \( e \in E(H) \) contains at least one vertex \( v_e \in e \) such that \( v_e \notin e' \) for all other hyperedges \( e' \in E(\H) \setminus \{e\} \), therefore, $\cup_{e \in A}e$ cannot be equal to $\cup_{e \in B}e$ and hence  $\lcm(A)\not= \lcm (B) \iff f(A)\not= f(B)$, as required. 
\end{proof}
\subsection{Hypergraphs and Modular Lattices}
\begin{definition}
     Let $\mathcal{L}$ be a lattice. For $y, z \in \mathcal{L}$, the pair $(y, z)$ is said to satisfy the \emph{modular law} if, for all $x \in \mathcal{L}$ with $x \leq z$, we have
     \[
     x \vee (y \wedge z) = (x \vee y) \wedge z.
     \]
     If every pair of elements in $\mathcal{L}$ satisfies the modular law, then $\mathcal{L}$ is called a \emph{modular lattice}.
\end{definition}
    Consider two $\lcm$-lattices: the pentagon lattice $N_5$ and the diamond lattice $M_3$, as shown in \ref{FigN5M3S}. The lattice $N_5$ is non-modular, while $M_3$ is modular but not distributive. Moreover, Theorem~4.10 of \cite{DaveyPristley2002} states as follows.
\begin{figure}[h]
   \centering
   \begin{minipage}{0.45\textwidth}
    \centering
    \begin{tikzpicture}[scale=1, line join=round, line cap=round]
    \tikzset{
    v/.style={circle, fill=blue!70, inner sep=2.2pt}, 
    edgea/.style={line width=1.2pt, teal},
    edgeb/.style={line width=1.2pt, orange!90!black},
    edgec/.style={line width=1.2pt, purple!80!black}
   }
  \node[v,label=above:{$\hat{1}$}]    (top) at (0,1.8) {};
  \node[v,label=left:{$a$}]    (left) at (-1,0.7) {};
  \node[v,label=right:{$y$}]   (right1) at (1,0.7) {};
  \node[v,label=right:{$x$}]   (right2) at (1,-0.3) {};
  \node[v,label=below:{$\hat{0}$}]   (bot) at (0,-1) {}; 
  \draw[edgea] (top)--(left);
  \draw[edgeb] (left)--(bot);
  \draw[edgec] (bot)--(right2);
  \draw[edgea] (right2)--(right1);
  \draw[edgeb] (right1)--(top);
\end{tikzpicture}
\end{minipage}%
\hfill
\begin{minipage}{0.55\textwidth}
   \centering
   \begin{tikzpicture}[scale=1.1, line join=round, line cap=round]
   \tikzset{
    v/.style={circle, fill=blue!70, inner sep=2.2pt}, 
    edgea/.style={line width=1.2pt, teal},
    edgeb/.style={line width=1.2pt, orange!90!black},
    edgec/.style={line width=1.2pt, purple!80!black}
   }
  \node[v,label=above:{$\hat{1}$}]    (top) at (0,1.8) {};
  \node[v,label=left:{$a$}]     (left) at (-1,0.5) {};
  \node[v,label=right:{$b$}]    (right) at (1,0.5) {};
  \node[v,label=right:{$c$}]    (mid) at (0,0.5) {};
  \node[v,label=below:{$\hat{0}$}]    (bot) at (0,-0.8) {}; 
  \draw[edgea] (top)--(left)--(bot)--(right)--(top);
  \draw[edgec] (top)--(mid)--(bot);
\end{tikzpicture}
\end{minipage}
\caption{The $N_5$ (left) and the $M_3$ (right) lattices.}\label{FigN5M3S}
\end{figure}
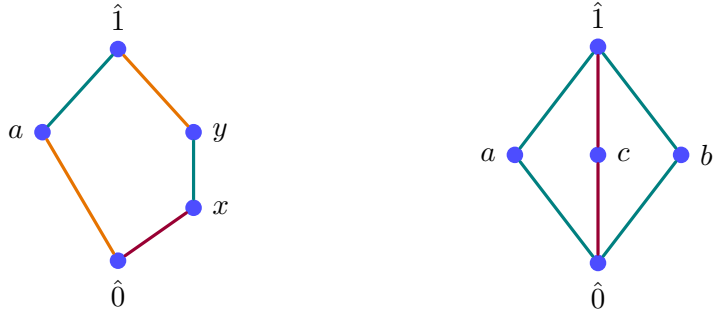
\begin{theorem}[Birkhoff's Characterization Theorem]
A lattice $\mathcal{L}$ is modular if and only if it does not contain a sublattice isomorphic to the pentagon lattice $N_5$.  
Moreover, $\mathcal{L}$ is distributive if and only if it contains no sublattice isomorphic to either the pentagon lattice $N_5$ or the diamond lattice $M_3$.
\end{theorem}

  The following result gives a characterization of modular in the case of hypergraphs.
\begin{theorem}\label{THypermodular} 
    Let $\mathcal{H}$ be a $k$-uniform hypergraph on $n$ vertices with $m>2$ hyperedges, and let 
    \[
    I(\mathcal{H})= (x_{i_1}\cdots x_{i_k} \mid \{i_1,\ldots,i_k\}\in E(\mathcal{H}))
    \]
    denote its corresponding edge ideal, such that at least one of the following conditions holds: 
\begin{itemize}
    \item[(a)] For every hyperedge \( e \in E(\mathcal{H}) \), there exists a vertex \( v_e \in e \) such that \( v_e \notin e' \) for all other hyperedges \( e' \in E(\mathcal{H}) \setminus \{e\} \)
    \item[(b)] \( k = n-1 \), i.e., each hyperedge has cardinality $n-1$.
\end{itemize}
   Then, the $\lcm$-lattice $\LC(I(\mathcal{H}))$ is modular.
\end{theorem}
\begin{proof}
\[ e_{i_1} \vee \cdots \vee e_{i_t} = m. \] 



If $(a)$ holds, then by \ref{Booleanhyper}, $\LC(I(\mathcal{H}))$ is Boolean, and therefore modular.  
If $(b)$ holds, then for each pair of hyperedges $e_i$ and $e_j$, there exists exactly one pair of vertices $v_i$ and $v_j$ such that $v_i \in e_i$ with $v_i \notin e_j$, and $v_j \in e_j$ with $v_j \notin e_i$. Consequently, the join of any two distinct atoms is the maximal element $\hat{1}$, while the meet of any two distinct atoms is $\hat{0}$. Hence, the lattice cannot contain a sublattice isomorphic to $N_5$, and is therefore modular.

\end{proof}
The following example illustrates different cases in connection to the above theorem.
\begin{example}\label{exampmodulr}
     Let $\mathcal{H}$ be a uniform hypergraph consisting of 3-uniform hyperedges $e_i$. Some cases are computed below to determine whether the corresponding $\lcm$-lattice $\mathcal{L}(I(\mathcal{H}))$ of the edge ideal is modular or not. In particular, we verify whether 
\[
x \vee (y \wedge z) \;=\; (x \vee y) \wedge z,
\]
for all $x, y, z \in \mathcal{L}(\mathcal{H})$ with $x \leq z$ holds, in the given cases.

Case 1: $e_1$ and $e_2$ share a common edge, while $e_2$ and $e_3$ share only a vertex, as shown in \ref{F3} along with the corresponding $\lcm$-lattice.

    For $e_1=\{1,2,3\}, e_2= \{2,3,4\}$ and $e_3=\{4,5,6\}$, we have $$I(\mathcal{H})=<x_1x_2x_3, x_2x_3x_4, x_4x_5x_6>,$$  
     Now, let
    \[
      x = x_1x_2x_3, 
      \qquad y = x_4x_5x_6, 
      \qquad z = x_1x_2x_3x_4.
    \]
\begin{figure}[h]
     \centering
     \begin{minipage}{0.44\textwidth}
     \centering
     \begin{tikzpicture}[line join=round, line cap=round, scale=0.64]
   \tikzset{
    v/.style={circle, fill=blue!70, inner sep=2.2pt},   
    vc/.style={circle, fill=red!80, inner sep=2.2pt}    
     }
  \coordinate (LT) at (0,1.5);
  \coordinate (LB) at (0,-1.5);
  \coordinate (C)  at (2,0);
  \coordinate (MT) at (4,1.5);
  \coordinate (MB) at (4,-1.5);
  \coordinate (R)  at (6,0);
  \fill[teal!30, opacity=0.65]   (LT) -- (LB) -- (C) -- cycle;   
  \fill[orange!30, opacity=0.65] (MT) -- (MB) -- (C) -- cycle;   
  \fill[purple!30, opacity=0.65] (MT) -- (R)  -- (MB) -- cycle;  
  \draw[thick] (LT) -- (LB) -- (C) -- cycle;
  \draw[thick] (MT) -- (MB) -- (C) -- cycle;
  \draw[thick] (MT) -- (R) -- (MB) -- cycle;
  \node[v]  at (LT) {};
  \node[v]  at (LB) {};
  \node[vc] at (C)  {};
  \node[v]  at (MT) {};
  \node[v]  at (MB) {};
  \node[v]  at (R)  {};
  \node[font=\small, above=1.2pt of LT]  {$5$};
  \node[font=\small, below=1.2pt of LB] {$6$};
  \node[font=\small, above=1.2pt of C]      {$4$};
  \node[font=\small, above=1.2pt of MT]       {$2$};
  \node[font=\small, below=1.2pt of MB] {$3$};
  \node[font=\small, above=1.2pt of R] {$1$};
   \end{tikzpicture}
\end{minipage}%
\hfill
\begin{minipage}{0.55\textwidth}
   \centering
   \begin{tikzpicture}[
   every node/.style={draw, rounded corners, minimum size=6mm, inner sep=3pt, font=\small},
  ->, >=Latex, scale=0.67, transform shape]
  \node[fill=blue!15] (A) at (0,4) {$\hat{1}$};
  \node[fill=orange!15] (B1) at (-2,3) {$x_1x_2x_3x_4$};
  \node[fill=purple!15] (B2) at ( 2,3) {$x_2x_3x_4x_5x_6$};
  \node[fill=teal!15]   (C1) at (-3,2) {$x_1x_2x_3$};
  \node[fill=orange!15] (C2) at ( 0,2) {$x_2x_3x_4$};
  \node[fill=purple!15] (C3) at ( 3,2) {$x_4x_5x_6$};
  \node[fill=gray!20] (O) at (0,0) {$\hat{0}$};
   \draw[->] (A) -- (B1);
  \draw[->] (A) -- (B2);
  \draw[->] (B1) -- (C1);
  \draw[->] (B1) -- (C2);
  \draw[->] (B2) -- (C2);
  \draw[->] (B2) -- (C3);
  \draw[->] (C1) -- (O);
  \draw[->] (C2) -- (O);
  \draw[->] (C3) -- (O);
  \end{tikzpicture}
  \end{minipage}
\caption{A hypergraph with triangular faces and its corresponding $\lcm$-lattice.}\label{F3}
\end{figure}
    Clearly, $x \leq z$ holds and:
    \begin{eqnarray}
    \nonumber x \vee (y \wedge z) &=& x_1x_2x_3 \vee(x_4x_5x_6\wedge x_1x_2x_3x_4)\\ 
    \nonumber   &=& x_1x_2x_3\vee \hat{0} =x_1x_2x_3,\\
    \nonumber (x \vee y) \wedge z &=& (x_1x_2x_3 \vee x_4x_5x_6)\wedge x_1x_2x_3x_4\\ 
   \nonumber  &=&  \hat{1}\wedge x_1x_2x_3x_4=x_1x_2x_3x_4.
\end{eqnarray}
   Thus, the lattice is \emph{not modular}.\\
    Case 2:
    The four triangles $e_1,e_2,e_3,e_4$ form the faces of a tetrahedron, where each triangular face $e_i$ shares an edge with each of the other three faces.  

    Let $e_1=\{1,2,3\},e_2=\{1,2,4\},e_3=\{1,3,4\},$ and $e_4=\{2,3,4\}$. 
     Then $$I(\mathcal{H})=<x_1x_2x_3, x_1x_2x_4, x_1x_3x_4,x_2x_3x_4>,$$ the corresponding $\lcm$-lattice is shown in \ref{F7}. 
     To verify $x \vee (y \wedge z) \;=\; (x \vee y) \wedge z$, for all $x,y$ and $z$ in $\mathcal{L}(I(\mathcal{H}))$, let
     \[
     x = x_1x_2x_3, 
     \qquad y = x_1x_2x_4, 
     \qquad z = \hat{1}.
     \]
     Clearly, $x \leq z$ holds and: 
\begin{figure}[h]
    \centering
    \begin{minipage}{0.45\textwidth}
    \centering
    \begin{tikzpicture}[line join=round, line cap=round, scale=1.35]
    \tikzset{
    v/.style={circle, fill=blue!70, inner sep=2.2pt},   
   vc/.style={circle, fill=red!80, inner sep=2.4pt},   
   }
   \coordinate (A) at (-1,-0.6);
   \coordinate (B) at ( 1,-0.6);
   \coordinate (C) at (0,1.2);
   \coordinate (D) at (0,0.2); 
   \fill[blue!20, opacity=0.7]   (A)--(B)--(C)--cycle;  
   \fill[teal!25, opacity=0.7]   (A)--(B)--(D)--cycle;  
   \fill[orange!25, opacity=0.7] (B)--(C)--(D)--cycle;  
   \fill[purple!25, opacity=0.7] (C)--(A)--(D)--cycle;  
   \draw[thick] (A)--(B)--(C)--cycle;  
   \draw[thick] (A)--(D);
   \draw[thick] (B)--(D);
   \draw[thick] (C)--(D);
  \node[v]  (nA) at (A) {};
  \node[v]  (nB) at (B) {};
  \node[v]  (nC) at (C) {};
  \node[vc] (nD) at (D) {};
  \node[below left=1pt of nA] {$3$};
  \node[below right=1pt of nB] {$4$};
  \node[above=1pt of nC] {$1$};
  \node[right=1pt of nD] {$2$};
  \end{tikzpicture}
\vspace{0.2cm}
\end{minipage}%
\hfill
\begin{minipage}{0.55\textwidth}
   \centering
   \begin{tikzpicture}[scale=1, every node/.style={draw, rounded corners, minimum size=5mm, inner sep=0.5pt, font=\small}, ->, >=Latex]
   \node[fill=blue!10] (T) at (0,3.5) {$\hat{1}$};
   \node[fill=blue!15]   (A) at (-2,2) {$x_1x_2x_3$};
   \node[fill=teal!15]   (B) at (-0.7,2) {$x_1x_2x_4$};
   \node[fill=orange!15] (C) at (0.7,2) {$x_2x_3x_4$};
   \node[fill=purple!15] (D) at (2,2) {$x_1x_3x_4$};
   \node[fill=gray!10] (O) at (0,0.5) {$\hat{0}$};
   \draw[->] (T)--(A);
   \draw[->] (T)--(B);
   \draw[->] (T)--(C);
   \draw[->] (T)--(D);
   \draw[->] (A)--(O);
   \draw[->] (B)--(O);
   \draw[->] (C)--(O);
   \draw[->] (D)--(O);
   \end{tikzpicture}
\vspace{0.2cm}
\end{minipage}
\caption{The hypergraph with triangular faces, and its corresponding $\lcm$- lattice}\label{F7}
\end{figure}
    \begin{eqnarray}
     \nonumber x \vee (y \wedge z) &=& x_1x_2x_3 \vee(x_2x_3x_4\wedge \hat{1})\\ 
     \nonumber   &=& x_1x_2x_3\vee x_2x_3x_4 =\hat{1},\\
     \nonumber (x \vee y) \wedge z &=& (x_1x_2x_3 \vee x_2x_3x_4)\wedge \hat{1}\\ 
     \nonumber  &=&  \hat{1}\wedge \hat{1}=\hat{1}.
     \end{eqnarray}
     This lattice has just two levels of atoms under the top and one bottom, it does not contain $N_5$ as a sublattice. Thus, the lattice is modular for any $x,y$, and $z$.
\end{example}

\subsection{Hypergraphs and Complemented Lattices}

\begin{definition}
    If $L$ is a bounded lattice then we say that $y \in L$ is a complement of $x \in L$ if $\exists ~ \hat{0}, \hat{1} $ such that  $x \wedge y = 0$ and $x \vee y = 1.$ In this case, we say that $x$ is a complemented element of $L$. Clearly, every complement of a complemented element is itself complemented. A lattice $L$ is a complemented lattice if every element of $L$ is complemented.
\end{definition} 


 

\begin{theorem}
Let $G$ be a connected graph and $I(G)$ be its edge ideal. If $G$ does not contain a path
\[
\{x_1,x_2,x_3,x_4\},
\]
such that $x_1$ and $x_4$ have degree $1$, then $\LC(I(G))$ is complemented.
\end{theorem}\begin{proof}

     Let $G$ does not contain any path of edges, say, $\{x_1,x_2\},$ $\{x_2,x_3\}, \{x_{3},x_4\}$ such that $x_1$ and $x_4$ have degree $1$. Let $u \in \mathcal{L}(I(G))$, with $\supp(u) =A\subseteq V(G)$, such that there exist a subset $E' \subset E(G)$ with $A=\cup_{e\in E'} e.$ Set $B =V(G) \setminus A $ and 
     $$B'=\{u \in B| u \text{ is not connected to another vertex in } B\}.$$ 
     Also, let $$C= B\setminus B' \cup \{\{u,w\}\in E(G) | u\in B'\}.$$
     Now there are two cases:

     Case $(i):$ $B'=\emptyset $. In this case, there exist $v \in \mathcal{L}(I(G))$ such that $\supp(v)=B=C$. It is easy to see, $v$ is the complement of $u$ in $\mathcal{L}(I(G))$.

     Case $(ii):$ $B'\not=\emptyset $. Since all vertices in $B'$ has degree $1$, then by supposition, as $G$ does not contain a paths, say, $\{x_1,x_2, \ldots, x_4\}$ such that $x_1$ and $x_4$ have degree $1$. Therefore, there exist $v \in \mathcal{L}(I(G))$  such that $\supp(v)=C$. But then we have $u \wedge v = \emptyset$, because there does not exist any common edge, between $x_2$ and $x_3$. Moreover, since $A \cup C = V(G)$, thus $v$ is the complement of $u$ in $\mathcal{L}$.
\end{proof}

\begin{theorem}
     Let $H$ be a $k$-uniform hypergraph and $I(H)$ its edge ideal. Then $\LC(I(H))$ is complemented if and only if $H$ does not contain a triplet of hyperedges $\{e_1, e_2, e_3\}$ such that $e_1$ and $e_3$ do not share vertices with any other hyperedge except $e_2$, where 
\[
e_1 \cap e_2 \neq \emptyset, \quad e_2 \cap e_3 \neq \emptyset, \quad \text{and} \quad |\,\{e_1, e_3\} \cap e_2\,| = |e_2|.
\]
\end{theorem}
\begin{proof}
    Let $\LC(I(H))$ be complemented and suppose that on contrary there exists a triplet of hyperedges, say $\{e_1, e_2. e_3\}$ such that it satisfies the given conditions. Now, let $u=e_2 \in \LC(I(H))$ be an element, then the complement $v$ of $u$ must contain $e_1\setminus e_2$ and $e_3\setminus e_2$. Then since $e_1 \cap e_2 \neq \emptyset, \quad e_2 \cap e_3 \neq \emptyset, \quad \text{and} \quad |\,\{e_1, e_3\} \cap e_2\,| = |e_2|$, this implies $e_2$ must belong to $v$, which contradics the definition of complement of an element. Therefore, $\LC(I(H))$ is not complemented, a contradiction.

    Conversely, there does not exist any triplet of hyperedges, say $\{e_1, e_2. e_3\}$ such that it satisfies the given conditions. Let $u$ be an element of $\LC(I(H))$ such that $A = \cup_{e\in u}e$ and let $B=V(H)\setminus A$. Now let 
    $$B'=\{u \in B \,\,|\,\, u \text{ is not connected to another vertex in } B\}.$$ 
    Also, let $$C= B\setminus B' \cup \{\{u,w\}\in E(G) | u\in B'\}.$$
    Now there are two cases:

    Case $(i):$ $B'=\emptyset $. In this case, there exist $v \in \mathcal{L}(I(G))$ such that $\supp(v)=B=C$. It is easy to see, $v$ is the complement of $u$ in $\mathcal{L}(I(G))$.

    Case $(ii):$ $B'\not=\emptyset $. Since all vertices in $B'$ do not share any vertex with any other vertex in B, then by supposition, as $G$ does not contain a paths, say, $\{x_1,x_2, \ldots, x_4\}$ such that $x_1$ and $x_4$ have degree $1$. Therefore, there exist $v \in \mathcal{L}(I(G))$  such that $\supp(v)=C$. But then we have $u \wedge v = \emptyset$, because there does not exist any common edge, between $x_2$ and $x_3$. Moreover, since $A \cup C = V(G)$, thus $v$ is the complement of $u$ in $\mathcal{L}$.

\begin{example}
Consider the lattice with elements
\[
\{0,\;12,\;13,\;24,\;123,\;124,\;1234\}.
\]

\begin{figure}[h]
    \centering
    \begin{minipage}{0.45\textwidth}
        \centering
        \begin{tikzpicture}[scale=0.8, line join=round, line cap=round]
            \tikzset{
                v/.style={
                    circle,
                    fill=blue!70,
                    inner sep=2.2pt,
                    font=\small,
                    text=white
                }
            }

            \node[v,label=below:{$4$}] (v1) at (2,0) {};
            \node[v,label=below:{$3$}] (v2) at (0,0) {};
            \node[v,label=above:{$2$}] (v3) at (0,2) {};
            \node[v,label=above:{$1$}] (v4) at (2,2) {};

            \draw[teal, thick] (v1)--(v2);
            \draw[orange!90!black, thick] (v2)--(v3);
            \draw[purple!80!black, thick] (v3)--(v4);
        \end{tikzpicture}
        \vspace{0.2cm}
    \end{minipage}%
    \hfill
    \begin{minipage}{0.45\textwidth}
        \centering
        \begin{tikzpicture}[
            every node/.style={
                rounded corners,
                draw,
                minimum size=6mm,
                font=\footnotesize
            },
            scale=0.6
        ]

            \node[fill=gray!10] (0) at (0,0) {$\hat{0}$};

            \node[fill=blue!15] (12) at (-3,1.5) {$12$};
            \node[fill=teal!15] (23) at (0,1.5) {$23$};
            \node[fill=orange!15] (34) at (3,1.5) {$34$};

            \node[fill=purple!15] (123) at (-1.5,3) {$123$};
            \node[fill=green!15] (234) at (1.5,3) {$234$};

            \node[fill=blue!10] (1) at (0,4.5) {$\hat{1}$};

            \draw (0)--(12) (0)--(23) (0)--(34);
            \draw (12)--(123) (23)--(123);
            \draw (23)--(234) (34)--(234);
            \draw (123)--(1) (234)--(1);

        \end{tikzpicture}
    \end{minipage}

    \caption{The path graph $P_4$ and its corresponding $\lcm$-lattice.}
    \label{F5}
\end{figure}

As illustrated in \ref{F5}, consider the $\lcm$-lattice
\[
L=\{0,\;12,\;13,\;24,\;123,\;124,\;1234\}.
\]
The elements $12$ and $34$ are complements of each other, since
\[
12\wedge 34=0
\qquad\text{and}\qquad
12\vee 34=1234.
\]
However, the element $23$ has no complement in $L$, since its complement would be $14$, which is not an element of the lattice.
Likewise, $123$ would require $4$ as a complement, while $234$ would
require $1$; neither $4$ nor $1$ belongs to $L$.

Hence, not every element of $L$ has a complement. Therefore, the
$\lcm$-lattice $L$ is not complemented.
\end{example}

\end{proof}


\begin{definition}
    A lattice $L$ is a \emph{relatively complemented lattice} if every interval $[x, y]$(viewed as a sublattice) of $L$ is complemented. 
\end{definition}

\begin{remark}
    In a distributive lattice, all complements and relatively complements of poset are unique. 
\end{remark}

\begin{theorem}
     Let $G$ be a connected graph and $I(G)$ be its edge ideal, then $lcm(I(G))$ is \emph{relatively complemented} if and only if there exists no induced subgraph $G|_W$ on the vertex set $W \subseteq V$ such that ${G|_W}$ is of path length $ \geq 3$.
\end{theorem}
\begin{proof}
     Suppose that $\lcm(I(G))$ is relatively complemented. 
To the contrary, assume that there exists an induced subgraph $W \subseteq V$ having a path of length $3$, where 
$W = \{x_1x_2,\, x_2x_3,\, x_3x_4\}$. 
Then, the interval 
\[
[a,\, \{x_1x_2x_3x_4\}]
\]
is not complemented for $a \in \{x_1x_2,\, x_3x_4\}$. \\[4pt]
Conversely, suppose that there exists no induced subgraph of path length $3$. 
This means that $G$ has a maximum induced path of length two. 
We now show that $\lcm(I(G))$ is relatively complemented. 
Since the $\lcm$-lattice of a graph having path length at most $2$ is Boolean, and every Boolean lattice is relatively complemented, we conclude that $\lcm(I(G))$ is relatively complemented.

\end{proof}
%
%
%
%
%
%
%
%
%
%
%
\subsection{Product of $\lcm$-Lattices}
We start with following basic property of the $\lcm$-lattices.
\begin{definition}
    Given two lattices $L_1$ and $L_2$, then the product of lattices $L_1 \times L_2$ forms a lattice in a natural way. The partial order is defined by:
\(
(x_1, y_1) \leq (x_2, y_2) 
\quad \text{if and only if} \quad
x_1 \leq x_2 \ \text{in } L_1 \ \text{and}\ y_1 \leq y_2 \ \text{in } L_2.
\)
\end{definition}

\begin{proposition}
     Let $L_1, \ldots, L_r$ be the lattices then $L_1 \times L_2 \ldots \times L_r$ be complemented if and only if $L_1, \ldots, L_r$ be complemented.
\end{proposition}

\begin{proof}
$(\Rightarrow)$  
Let $L=L_1 \times L_2 \times \cdots \times L_r$ be complemented. Let $x_i \in L_i$ and consider the element $a = (0_1, \dots, 0_{i-1}, x_i, 0_{i+1}, \dots, 0_r) \in L$.  
Since $L$ is complemented, there exists $a' = (y_1, \dots, y_r) \in L$ such that  
$$a \vee a' = 1_L \quad \text{and} \quad a \wedge a' = 0_L.$$  
Now the $i$-th component gives  
$$x_i \vee y_i = 1_{L_i} \quad \text{and} \quad x_i \wedge y_i = 0_{L_i}.$$  
Thus $y_i$ is a complement of $x_i$ in $L_i$. Since $x_i$ was arbitrary, $L_i$ is complemented.

$(\Leftarrow)$  
Now let $L_i$ is complemented for all $i$. Let $u = (x_1, x_2, \dots, x_r) \in L$.  
Now for each $x_i\in L_i$, there exist a complement $x_i' \in L_i$ such that:  
$$x_i \vee x_i' = 1_{L_i} \quad \text{and} \quad x_i \wedge x_i' = 0_{L_i}.$$  
Set $u' = (x_1', x_2', \dots, x_r') \in L$.  
Then from above:  
$$u \vee u' = (x_1 \vee x_1', \dots, x_r \vee x_r') = (1_{L_1}, \dots, 1_{L_r}) = 1_L$$  
and  
$$u \wedge u' = (x_1 \wedge x_1', \dots, x_r \wedge x_r') = (0_{L_1}, \dots, 0_{L_r}) = 0_L.$$  
Therefore $u'$ is a complement of $u$ in $L$. 

\end{proof}

\subsection{Impact of Polarization on $\lcm$-lattices}\label{subsec2}
In the above discussion, we have considered it $I$ to be a square-free monomial ideal. To extend these results to the case of a general monomial ideal $I$ in the polynomial ring $S$, we use the concept of \emph{polarization}, as defined below. In particular, we aim to show that the $\lcm$-lattice of $I$ and that of its polarization $I^{p}$ are isomorphic. Consequently, all the above results will also hold for any monomial ideal $I$.

\begin{definition}

    Let $S = \mathbb{K}[x_1, x_2, \ldots, x_n]$ be a polynomial ring over a field $\mathbb{K}$, and let 
   \[
    I = (u_1, u_2, \ldots, u_m) \subseteq S
   \] 
    be a monomial ideal, where each $u_j = x_1^{a_{1j}} x_2^{a_{2j}} \cdots x_n^{a_{nj}}$ is a monomial in $S$. The \emph{polarization} of $I$, denoted $I^{p}$, is a square-free monomial ideal obtained in a larger polynomial ring 
    \[
    S^{p} = \mathbb{K}[x_{i1}, x_{i2}, \ldots, x_{i a_i} \mid 1 \leq i \leq n],
    \] 
     where $a_i = \max_j \{a_{ij}\}$. Each monomial $u_j$ is replaced by the square-free monomial
    \[
     u_j^{p} = \prod_{i=1}^n \prod_{k=1}^{a_{ij}} x_{ik}.
    \]
    Then, 
    \[
     I^{p} = (u_1^{p}, u_2^{p}, \ldots, u_m^{p}).
    \]
\end{definition}
\begin{example}
    Consider the polynomial ring $S = \mathbb{K}[x_1, x_2]$ and the monomial ideal
   \[
    I = (x_1^2 x_2, \; x_2^3).
    \]
    To polarize, we introduce new variables $x_{11}, x_{12}$ for $x_1^2$ and $x_{21}, x_{22}, x_{23}$ for $x_2^3$. Thus, in 
    \[
     S^{p} = \mathbb{K}[x_{11}, x_{12}, x_{21}, x_{22}, x_{23}],
    \]
    the polarized ideal is
    \[
    I^{p} = (x_{11}x_{12}x_{21}, \; x_{21}x_{22}x_{23}).
    \]
    Here, $I^{p}$ is square-free, and it preserves many algebraic properties of $I$.
    So, for $I$, polarization gives $ I^{p}$, preserving the $\lcm$-lattice structure.
\end{example}   
 The following result shows that polarization preserves the $\lcm$-lattice isomorphism between monomial and square-free ideals. This enables us to extend our characterizations of Boolean and modular properties to any monomial ideal.

\begin{lemma}\label{Tpolarization}
    Let \( I \subset S = \mathbb{K}[x_1, \ldots, x_n] \) be a monomial ideal and let \( I^p \) denote its polarization. Let \( \mathcal{L}(I) \) and \( \mathcal{L}(I^p) \) be the $\lcm$-lattices of \( I \) and \( I^p \), respectively. Then,
    \[
    \mathcal{L}(I) \cong \mathcal{L}(I^p).
    \]
\end{lemma}

\begin{proof}
    The $\lcm$-lattice of a monomial ideal is atomic, and two monomials are equal if and only if their polarizations are equal. It suffices to show that the polarization of the join of two monomials equals the join of their polarizations; that is $(m_1 \vee m_2)^p = m_1^p \vee m_2^p$.
    Let \( m_1 = x_1^{\alpha_1} \cdots x_n^{\alpha_n} \) and \( m_2 = x_1^{\beta_1} \cdots x_n^{\beta_n} \). Then,
    \[
    m_1 \vee m_2 = x_1^{\max(\alpha_1, \beta_1)} \cdots x_n^{\max(\alpha_n, \beta_n)}.
    \]
    The polarization of \( m_1 \) and \( m_2 \) are given by:
    \[
    m_1^p = \prod_{i=1}^n \prod_{j=1}^{\alpha_i} x_{ij}, \quad
    m_2^p = \prod_{i=1}^n \prod_{j=1}^{\beta_i} x_{ij}.
    \]
    Therefore,
    \[
    m_1^p \vee m_2^p = \prod_{i=1}^n \prod_{j=1}^{\max(\alpha_i, \beta_i)} x_{ij},
    \]
    which is exactly the polarization of \( m_1 \vee m_2 \). Hence,
    \[
    (m_1 \vee m_2)^p = m_1^p \vee m_2^p.
    \]
    This shows that the $\lcm$-lattices \( \mathcal{L}(I) \) and \( \mathcal{L}(I^p) \) are isomorphic.
\end{proof}





%

\end{document}